\documentclass[a4paper,11pt]{article}
\usepackage{graphicx}
\usepackage[margin=2.4cm]{geometry}
\usepackage{amsthm}

\usepackage{mathtools}

\usepackage{amsmath}

\usepackage{enumerate}

\usepackage{amsfonts}
\usepackage{amscd} 

\usepackage{subcaption}

\usepackage[backref]{hyperref}

\newcommand{\EE}{\mathbb{E}}
\newcommand{\PP}{\mathbb{P}}
\newcommand{\RR}{\mathbb{R}}
\newcommand{\NN}{\mathbb{N}}
\newcommand{\eps}{\varepsilon}
\newcommand{\defas}{\coloneqq}
\newcommand{\safed}{\eqqcolon}
\newcommand{\m}[1]{\mathbb{#1}}

\usepackage{todonotes}

\newtheorem{theorem}{Theorem}

\newtheorem{proposition}{Proposition}
\newtheorem{lemma}{Lemma}
\newtheorem{remark}{Remark}
\newtheorem{example}{Example}

\theoremstyle{definition}
\newtheorem{definition}{Definition}

\usepackage{cleveref} 

\begin{document}

\title{ 
    Zero-sum Random Games on Directed Graphs
}

\author{ 
    Luc Attia\thanks{Paris Dauphine University, France.}
    \and
    Lyuben Lichev\thanks{Univ.Jean Monnet, Saint-Etienne, France and Institute of Mathematics and Informatics, Bulgarian Academy of Science, Sofia, Bulgaria.}
    \and 
    Dieter Mitsche\thanks{Univ.~Jean Monnet, Saint-Etienne, France and Institute for Mathematical and Computational Engineering, Pontif\'icia Universidad Cat\'olica, Chile.}
    \and
    Raimundo Saona\thanks{Institute of Science and Technology Austria, Austria.}
    \and 
    Bruno Ziliotto\thanks{CEREMADE, CNRS, Paris Dauphine University, France.}
}

\date{\today}

\maketitle

\begin{abstract}
    This paper considers a class of two-player zero-sum games on directed graphs whose vertices are equipped with random payoffs of bounded support known by both players.
    Starting from a fixed vertex, players take turns to move a token along the edges of the graph. 
    On the one hand, for acyclic directed graphs of bounded degree and sub-exponential expansion, we show that the value of the game converges almost surely to a constant at an exponential rate dominated in terms of the expansion.
    On the other hand, for the infinite $d$-ary tree that does not fall into the previous class of graphs, we show convergence at a double-exponential rate in terms of the expansion.
\end{abstract}

\section{Introduction}
The following class of two-player zero-sum games has been introduced in \cite{garnier2021PercolationGames} under the name of \textit{percolation games}. 
Each vertex of $\m{Z}^d$ is equipped with a real-valued random variable called \textit{payoff}.
The realization of all these variables is known to the players at the start of the game. 
Initially, a token is placed at some vertex of $\m{Z}^d$ and, at every stage, each player chooses an action.
Then, the token is moved consecutively by Player 1 and by Player 2 according to the chosen actions. 
At the end of every stage, Player 2 pays to Player 1 the payoff of the corresponding vertex. 
Player 1 aims at maximizing the mean payoff over $n$ stages, while Player 2 aims at minimizing the same quantity, and the value of that game is denoted by $V_n$. 
The main result of \cite{garnier2021PercolationGames} shows that, when payoffs are bounded i.i.d.\ random variables and the game is \textit{oriented} (meaning that, at every move, the projection of the position of the token onto some fixed axis increases), then $(V_n)$ converges almost surely (a.s.) to a constant. 



The class of percolation games is motivated by various reasons.
First, it relates to the rich game-theoretic literature on the existence of a limit value in dynamic games (see for example the surveys \cite{LS15, SZ16}). This topic is particularly delicate for dynamic games with infinite state space, where general positive results are scarce (see~\cite{garrec19, LR20, ziliotto21} for some recent advances, and \cite{ziliotto13} for several counterexamples).
Second, percolation games connect to the important topic of stochastic homogenization of partial differential equations, for example, see~\cite[Section 4]{garnier2021PercolationGames} for results on Hamilton-Jacobi equations.
Moreover, from a probabilistic point of view, percolation games combine aspects of first-passage and last-passage percolation \cite{ADH17}, and a related model of Probabilistic Finite Automaton has been studied in \cite{HMM19,BKPR23}. 
Finally, it contributes to the growing literature on random games (see e.g.\ \cite{FPS23,ARY21,ACSZ21,HJMP21}). 

In this paper, we consider a model with a structure similar to that of a percolation game but where the state space is not restricted to be the graph $\m{Z}^d$. We introduce \textit{directed games} where the state space of the game is the vertex set of an acyclic directed graph $\Gamma$ where players move the token along the edges of $\Gamma$ respecting their orientation.
On the one hand, under certain assumptions of transitivity and sub-exponential growth of $\Gamma$, 
we prove that $V_n$ is exponentially concentrated around a given deterministic value (so, in particular, converges to that value a.s.) and relate the convergence rate to the speed of growth of the graph.
On the other hand, we consider the infinite $d$-ary tree with $d \geq 2$ where each vertex has exactly $d$ children and every edge is directed from the parent to the child. 
These graphs do not belong to the previous class of transient games due to their exponential growth. 
In this case, we show a stronger double-exponential concentration of $V_n$ around its expected value.


\section{Preliminaries}

A \emph{directed game} is a dynamical system that consists of a locally finite directed graph $\Gamma$ with infinite countable vertex set $Z$ called the \emph{state space}, an initial state $z_0 \in Z$ and a collection of independent and identically distributed (i.i.d.)\ random variables $(G_z)_{z\in Z}$ called \emph{payoffs}. 
We assume that $\Gamma$ has uniformly bounded degrees and contains neither directed cycles nor vertices with out-degree~0. 
The game is played by two players called Player~$1$ and Player~$2$. 
At the start of the game, the payoffs $(G_{z})_{z \in Z}$ are sampled and presented to both players, who thus obtain perfect information. 
Then, a token is placed at the initial state $z_0$.
For every integer $i\ge 0$, given that the token is positioned at a state $z\in Z$ before stage $i+1$, the following happens:
\begin{itemize}
    \item if $i$ is even, Player $1$ moves the token to an out-neighbor $z'$ of $z$ in $\Gamma$,
    \item if $i$ is odd, Player $2$ moves the token to an out-neighbor $z'$ of $z$ in $\Gamma$,
    \item Player 1 receives the payoff $G_{z'}$ from Player $2$.
\end{itemize}

Note that, unlike the setting in~\cite{garnier2021PercolationGames}, only one of the players performs a move at each round.
We are mostly interested in the \emph{$n$-stage game} consisting of the first $n$ stages for (typically large) integers $n$.

A \emph{strategy} of Player $1$ (resp. Player $2$) is a function $\sigma \colon \bigcup_{m \geq 0} Z^{2m+1} \rightarrow Z$ (resp. $\tau \colon \bigcup_{m \geq 0} Z^{2m+2} \rightarrow Z$) with the property that, for every $m \geq 0$ and $(z_0,z_1,\dots,z_{2m+1})\in Z^{2m+2}$, $\Gamma$ contains the edge from $z_{2m}$ to $\sigma(z_0,\dots,z_{2m})$ (resp. from $z_{2m+1}$ to $\tau(z_0,\dots,z_{2m+1})$). 
We denote by $\Sigma$ the collection of all strategies for Player $1$ and by $\mathcal{T}$ the collection of all strategies for Player $2$.

Given a pair of strategies $(\sigma, \tau)\in \Sigma\times \mathcal{T}$, we define inductively the \emph{trajectory of the token} by setting $z_{2i+1} \defas \sigma(z_0,\dots,z_{2i})$ and $z_{2i+2} \defas \tau(z_0,\dots,z_{2i+1})$ for every $i\ge 0$. 
This allows us to define the \emph{$n$-stage payoff function} $\gamma^{z_0}_n \colon \Sigma \times \mathcal{T} \to \RR$ by setting 
    \[
    \gamma^{z_0}_n(\sigma, \tau) \defas \frac{1}{n} \sum_{i=1}^{n} G_{z_i}. 
    \]

Note that, for a fixed initial state, since the directed graph $\Gamma$ is locally finite, the $n$-stage game is played on a finite state space with perfect information. 
For every $z\in Z$, the \emph{$n$-value} of the game with initial state $z_0 = z$ is defined as
    \[
    V_n(z) 
        \defas \max_{\sigma \in \Sigma} \min_{\tau \in \mathcal{T}} \gamma^{z}_n(\sigma, \tau)
        = \min_{\tau \in \mathcal{T}} \max_{\sigma \in \Sigma} \gamma^{z}_n(\sigma, \tau) \,,
    \]
where the classic minimax theorem~\cite{DP95, Neu28} applies to justify the above equality.
Moreover, we will say that a strategy $\sigma\in \Sigma$ (resp. $\tau\in \mathcal T$) is \emph{optimal} for the $n$-stage game (starting from $z$) if $\sigma$ maximizes $\min_{\tau\in \mathcal{T}} \gamma_n^z(\,\cdot\, , \tau)$ over $\Sigma$ (resp. if $\tau$ minimizes $\max_{\sigma\in \Sigma} \gamma_n^z(\sigma, \,\cdot\, )$ over $\mathcal{T}$).

A classic question in the game-theoretic literature is to ask for the convergence of the $n$-value as $n$ grows to infinity.
Since payoffs are random, $V_n$ is a random variable. Therefore, we are interested in whether the sequence $(V_n)$ converges a.s.\ to a constant.

In our model, if no further assumptions are imposed, it is possible that $(V_n)$ does not converge. 
For example, for all integers $m\ge 0$, set $n_m \defas 2^{2^{2m}}$ and $n'_m \defas 2^{2^{2m+1}}$ and consider the case where $\Gamma$ is a directed tree (all edges being directed away from the root) where each node with even height has only one child, 
while each node with odd height $k$ has two children if $k=1$ or $k\in [n_m,n_m')$ for some $m\ge 0$, and it has only one child if $k\in [n_m',n_{m+1})$.
Moreover, let the payoffs be i.i.d.\ Bernoulli random variables with parameter $1/2$.
In particular, for every $m\ge 1$, in the $n_m$-stage game, Player 2 has only one choice most of the time, while in the $n_m'$-stage game, she has two choices most of the time. 
Since Player 2 can not uniformly pick a vertex with payoff 0 (if it is present), and pick an available vertex otherwise, one can show that a.s.
\[
    \limsup_{m\to \infty} V_{n_m'}\le \frac{3}{8} < \frac{1}{2} = \lim_{m\to \infty} V_{n_m} \,.
\]
Indeed, while Player~1 never has a choice in the $n_m'$-game (implying that the mean payoff over the odd states visited by the token a.s.\ converges to $1/2$), Player~2 can ensure with the above strategy that the mean payoff over the even states visited by the token a.s.\ converges to $1/4$, which yields that a.s.\ $\limsup_{m\to \infty} V_{n_m'}\le 3/8$. 
At the same time, for every $\eps > 0$, Chernoff's bound for the Binomial distribution $\mathrm{Bin}(n_m, 1/2)$ and a union bound over the $O(2^{n_{m-1}'})$ vertices at level $n_m$ in $\Gamma$ shows that $V_{n_m}$ is in the interval $[1/2-\eps, 1/2+\eps]$ with probability very close to 1.
In particular, a.s.\ $(V_n)$ does not converge.
Therefore, to ensure convergence, we will need further structural assumptions on the graph.
\\
Before turning to our results, we provide some vocabulary. 
Given a vertex $z\in Z$, a \emph{descendant} of $z$ (in $\Gamma$) is a vertex that can be reached from $z$ by a directed path in $\Gamma$.
We say that $z$ and $z'$ are \textit{equivalent} 
if the two subgraphs of $\Gamma$ induced by the descendants of $z$ and by the descendants of $z'$, respectively, are isomorphic (as directed graphs).

\begin{definition}\label{def:wt}
    The graph $\Gamma$ is \textit{weakly transitive} if there is a state $z^*$ and an integer $M$ such that the following holds: for each state $z\in Z$, in the game with initial state $z_0 = z$, each player has a strategy that, independently of the moves of the opponent, ensures that the token is placed at a state equivalent to $z^*$ after an even number of $\ell\le M$ stages.
\end{definition}

Note that all vertex-transitive graphs are weakly transitive with $M = 0$.
In the remainder of the paper, we always assume that $\Gamma$ is weakly transitive. The next two subsections present two types of directed games used in our main results. 

\subsection{Weakly transitive games with sub-exponential expansion}

We continue with a few definitions. 
Given a state $z\in Z$, we consider a partition $\Pi_z \defas (Z_i(z))_{i\ge 0}$
of $Z$ such that: (i) $Z_0(z) = \{z\}$; and~(ii), for all strategies $(\sigma, \tau)\in \Sigma\times \mathcal{T}$ for the game starting at $z$ and for all $i \ge 1$, the token can visit the set $Z_i(z)$ at most once. 
Since $\Gamma$ has no directed cycles, such a partition exists. 
For example, the trivial one where every part contains a single state satisfies this property.
We call such partitions \textit{adapted}.  
For every integer $n\ge 1$, we also set $Z_{[n]}(z) \defas \bigcup_{j=0}^n Z_j(z)$ and $Z^{(n)}(z)$ for the set of \emph{reachable states} from $z$ after at most $n$ steps. 
Note that, when it is clear from the context, we omit $z$ from the notation and simply write $Z_n, Z^{(n)}$ and $Z_{[n]}$ for better readability.

Given a family of adapted partitions $\Pi \defas (\Pi_z)_{z\in Z}$ in a directed game, we define the \emph{transient speed function} $h$ of $\Pi$ as
\[
    h \colon n \in \NN \mapsto \max_{z\in Z} \min \left\{ k \in \NN : Z^{(n)}(z) \subseteq Z_{[k]}(z) \right\}.
\]
\noindent
Note that $h(n) \ge n$ for every integer $n\ge 1$ since, for every $z \in Z$, exactly $n$ of the sets $(Z_i(z))_{i\ge 1}$ are visited by the token after $n$ stages.
Our main goal is to analyze directed games where the size of the sets $Z^{(n)}(z)$ does not increase too fast as $n$ grows to infinity.

\begin{definition}[$\delta$-transient games]\label{def1}
    Given a family of adapted partitions $\Pi$ with transient speed $h$, we define the function $\psi \colon \NN \times (0, \infty) \to \RR$ by
    \[
        \psi(n, t) \defas \exp \left( -\frac{t^2 n^2}{2h(n)} \right) \max_{z\in Z} |Z^{(2n)}(z)| \,.
    \] 
    For a fixed $\delta > 0$, a directed game on a graph $\Gamma$ with vertex set $Z$ is called \emph{$\delta$-transient} if there exists a family of adapted partitions $\Pi$ of $Z$ and a sequence $(\eps_n)_{n\ge 1}$ such that $\eps_n + \psi(n, \eps_n) = O(n^{-\delta})$. Such a family $\Pi$ is called a \textit{$\delta$-adapted family}. 
\end{definition}

\begin{remark}
    The concept of $\delta$-transient games is only relevant for $\delta\in (0, 1/2)$.
    Indeed, Definition~\ref{def1} requires that $(\psi(n, \eps_n))_{n}$ converges to zero. 
    Therefore, since $h(n) \ge n$, this implies that $n = o(\eps_n^2 n^2)$, so that $\eps_n \in o(n^{1-1/2})$.
\end{remark}

\begin{remark}
    A sufficient condition under which a directed game is $\delta$-transient is the following: there exists an adapted partition $\Pi$ and real numbers $\alpha\in [0,2-2\delta)$ and $\beta\in [0,2-2\delta-\alpha)$ such that $h(n) = O(n^{\alpha})$ and $\max_{z\in Z} |Z^{(n)}(z)| = \exp(O(n^{\beta}))$.
\end{remark}

Note that the definition of a $\delta$-transient game is independent of the payoffs and only makes assumptions on the state space and the associated adapted partition. 
We now give a few examples of $\delta$-transient games.

\subsubsection{Oriented directed games}
Fix an integer $d\ge 1$, and denote by $e_i$ the $d$-dimensional vector with 1 in coordinate $i$ and 0 in all other $d-1$ coordinates.
Given positive integers $n_1,\ldots,n_d\ge 1$, a (directed) graph $\Gamma$ with vertex set $Z\subseteq \mathbb Z^d$ is called \emph{$(n_1,\ldots,n_d)$-invariant} (or simply \emph{invariant}) if, for every $i \in [1, d]$, the translation at vector $n_i e_i$ is a graph isomorphism for $\Gamma$.
A directed game is called \emph{oriented} if its underlying graph $\Gamma$ is invariant and there exists $u \in \RR^d\setminus \{0\}$ such that, for every directed edge $zw$ in $\Gamma$, we have $(w-z) \cdot u > 0$ (here, $\cdot$ denotes the usual scalar product of vectors in $\mathbb R^d$). We show the following proposition. 

\begin{proposition}\label{prop:oriented}
    Every oriented directed game is $\delta$-transient for all $\delta\in (0,1/2)$.
\end{proposition}

The following two classes of games present particular examples of oriented directed games.

\begin{example}[Games on tilings]
    A \emph{tiling} is a periodic partition of the plane into translations of one or several polygonal shapes, called \emph{tiles}, with vertices in $\mathbb Z^2$. 
    Tilings naturally define planar graphs whose vertex set coincides with the corners of the tiles and two vertices are connected by an edge if these can be connected by following the boundary of a tile without meeting another vertex on the way.
    By equipping the edges of this graph with suitable orientations, one can generate many different oriented directed games, see e.g.\ Figure~\ref{fig:tiling}. 
    
    \begin{figure}[ht!]
        \centering 
        \includegraphics[scale=0.3]{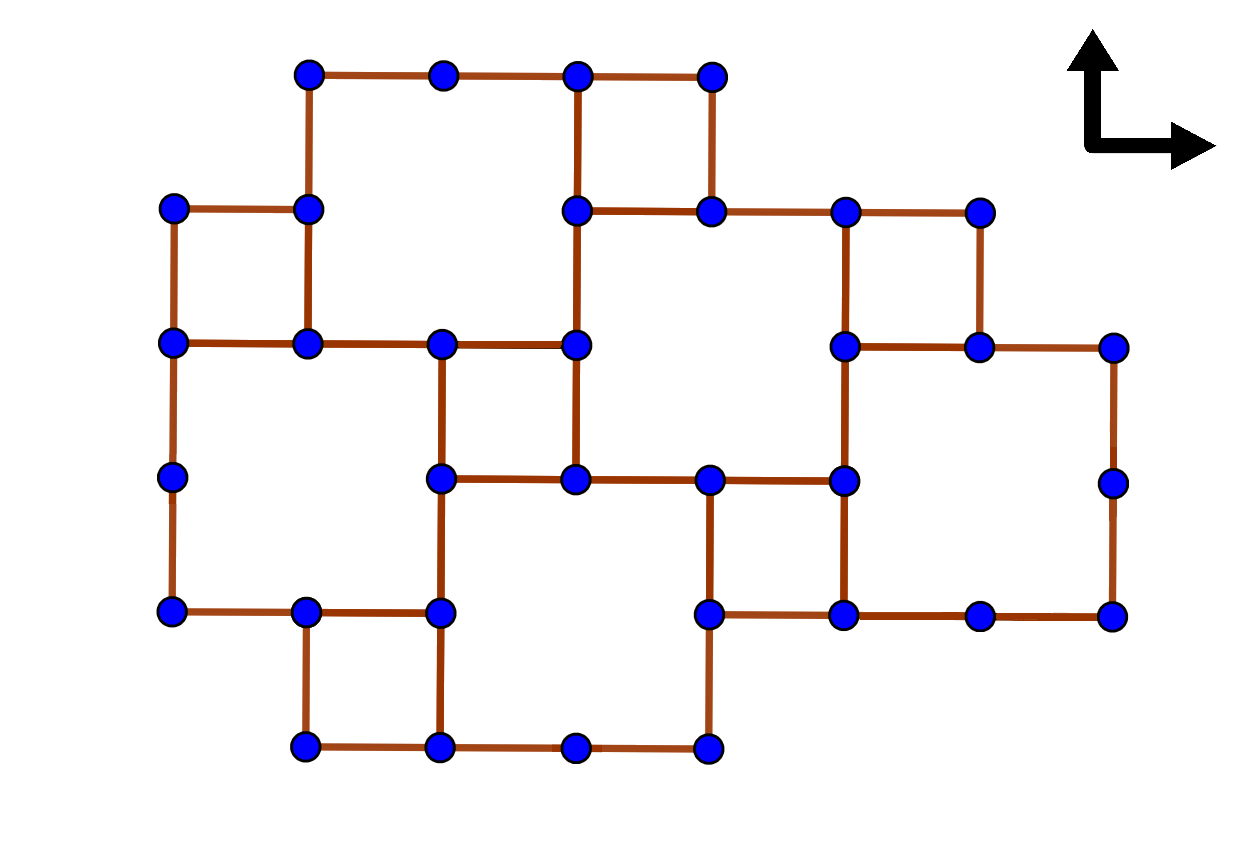}
        \caption{The figure depicts part of a tiling with two types of square tiles. 
        The vertices and the edges of the planar graph originating from the tiling are depicted in blue and red, respectively.
        Each horizontal edge is oriented from left to right and every vertical edge is oriented from bottom to top. 
        One may choose $z^*$ to be the bottom left vertex of a small square and $M=6$.}
        \label{fig:tiling}
    \end{figure}    
\end{example}

\begin{example}[Games on directed chains of graphs]
    Fix a finite vertex-transitive graph $H$ with vertex set $V(H)$ and edge set $E(H)$, and a bi-infinite sequence of copies $(H_i)_{i\in \mathbb Z}$ of $H$. 
    For every $i\in \mathbb Z$ and $u\in V(H)$, denote by $u_i$ the vertex in $H_i$ corresponding to $u$.
    We call an \emph{$H$-chain} the graph $\Gamma_H$ with vertices $\bigcup_{i\in \mathbb Z} V(H_i)$ and edges $\{u_iv_{i+1}: i\in \mathbb Z, uv\in E(H)\}$.
    
    Games on $H$-chains can be seen as instances of oriented directed games on $\mathbb Z$.
    Indeed, fixing $h = |V(H)|$, one may identify the vertices of $H_i$ with the integers in the interval $[ih+1, (i+1)h]$ for all $i\in \mathbb Z$ in a translation-invariant way.
\end{example}

\subsubsection{Weakly transitive games with controlled expansion}

Fix an arbitrary infinite rooted tree $T$ with root $r$ and a family of vertex-disjoint infinite paths $(P_v)_{v\in V(T)}$ where the path $P_v$ starts at vertex $v$ in $T$.
Define $\Gamma = T\cup (\bigcup_{v\in V(T)} P_v)$ as the tree rooted in $r$ and with all edges oriented away from $r$.
Let $(Z_n)_{n\ge 0}$ be a partition of the vertex set $Z$ of $\Gamma$ where $Z_n$ consists of all vertices at distance $n$ from $r$ for all $n\ge 0$.
Also, for every $z\in Z$ and $k\ge 2$, define $Z_k(z)$ to be the set of descendants of $z$ at distance $k$ from it while $Z_0(z) = \{z\}$ and $Z_1(z) = Z\setminus ((\bigcup_{k\ge 2} Z_k(z))\cup Z_0(z))$.
Note that, somewhat arbitrarily, we added all vertices not reachable from $z$ to $Z_1(z)$ to ensure that $(Z_i(z))_{i\ge 0}$ is a partition of $Z$.
Then, $\Pi_z = (Z_i(z))_{i\ge 0}$ is an adapted partition and $\Pi = (\Pi_z)_{z\in Z}$ is an adapted family of partitions.
Moreover, a single move of each player is sufficient to place the token at the second vertex of some infinite path among $(P_v)_{v\in V(T)}$.
This implies that the game is weakly transitive. 

Let us show that we can control the growth speed of $\max_{z\in Z} |Z^{(2n)}(z)|$.
Consider a set of non-negative integers $L = \{\ell_i: i\ge 1\}$ with $\ell_1 < \ell_2 < \ldots$ and let every vertex of $T$ in level $\ell$ have two children if $\ell\in L$ and one child otherwise. Moreover, suppose that $\ell_1 = 0$ and $(\ell_i - \ell_{i-1})_{i\ge 1}$ is a non-decreasing sequence.
Then, one can readily check that, for every $n\ge 1$, $\max_{z\in Z} |Z^{(n)}(z)| = |Z^{(n)}(r)|$. 
Indeed, for every $k, n\ge 1$ and a vertex $z\in Z$ on level $k$, using the assumptions that $\ell_1 = 0$ and $(\ell_i - \ell_{i-1})_{i\ge 1}$ is a non-decreasing sequence, we get
\[|Z^{(n)}(z)\setminus Z^{(n-1)}(z)| = 2^{|L\cap \{k,\ldots,k+n-1\}|}\le 2^{|L\cap \{0,\ldots,n-1\}|} = |Z^{(n)}(r)\setminus Z^{(n-1)}(r)|.\]
Thus, for every integer $n\ge 0$, $|Z^{(n)}(r)| = 1+\sum_{i=0}^{n-1} 2^{|L\cap \{0,\ldots,i\}|}$. Therefore, by a suitable choice of the set $L$, one can construct a tree $T$ with an arbitrary growth that is faster than linear but slower than exponential. 
In particular, for every $\delta\in (0, 1/2)$, this shows the existence of games that are $\delta$-transient but, for every $\delta' > \delta$, not $\delta'$-transient.

\subsection{\texorpdfstring{Directed games on $d$-ary trees}{}}

We turn our attention to a natural example of a directed game where the set of reachable states after $n$ steps grows exponentially with $n$.
Note that, for all $\delta > 0$, it is not a $\delta$-transient game. 
Fix an integer $d\ge 2$ and let $T$ be an \emph{infinite $d$-ary tree}, that is, a tree where every vertex has $d$ children, with vertex set $Z$ where every edge is oriented from the parent to the child.
We fix an arbitrary initial vertex $z_0$ and, for every integer $i\ge 0$, we define $Z_i$ to be the set of vertices in $Z$ that can be reached from $z_0$ by exactly $i$ steps and also denote $Z_{\rm{even}} \defas \bigcup_{i\ge 0} Z_{2i}$ and $Z_{\rm{odd}} \defas \bigcup_{i\ge 0} Z_{2i+1}$. 
Note that, for every $n \geq 1$, the random variables $(V_n(z))_{z \in Z}$ have the same distribution. 
Thus, we often omit the dependence of $V_n$ in $z$. 

\subsection{Main results}

Our first main result shows sharp concentration for the $n$-value of $\delta$-transient games around a deterministic constant.

\begin{theorem}\label{thm:1}
    Fix $\delta\in (0,1/2)$. Consider a $\delta$-transient directed game, a $\delta$-adapted family with transient speed $h$, and i.i.d.\ payoffs $(G_z)_{z\in Z}$ supported on the interval $[0,1]$. 
    Then, there exist constants $v_\infty \in [0, 1]$ and $K > 0$ such that, for all $n\ge 1$, $t\ge 0$, and $z \in Z$,
    \[
        \PP \left( |V_n(z) - v_\infty| \ge  t + K n^{-\delta} \right) \le 2 \exp \left(-\frac{t^2 n^2}{2h(n)} \right) \,.
    \]
    Consequently, $(V_n)$ converges almost surely to $v_\infty$.
\end{theorem}

Our second main result shows that the $n$-value of the directed game on a $d$-ary tree is tightly concentrated around a constant.
\begin{theorem}
\label{thm:2}
    Fix an integer $d\ge 2$. Consider a directed game on the $d$-ary tree with i.i.d.\ payoffs supported on the interval $[0, 1]$. 
    Then, there exists a real number $v_\infty \in [0, 1]$ such that, for every $\delta \in (0,1/2)$, there exists $K > 0$ such that, for every $n\ge 1$ and $t \ge 0$,
    \[
        \PP(|V_n - v_\infty| \ge t + 2t^2 + Kn^{-\delta}) \le \exp\left(-\frac{1}{6}\exp\left(\frac{t^2n}{4}\right)\right) \,.
    \]
    Consequently, $(V_n)$ converges almost surely to $v_\infty$.
\end{theorem}



\paragraph{Outline of the proofs.} The proofs of both theorems contain two main steps. 
The first step involves standard concentration arguments showing that $V_n$ is close to $\mathbb E[V_n]$ with high probability.
While these are sufficient for Theorem~\ref{thm:1}, the stronger probabilistic bound in Theorem~\ref{thm:2} requires an additional boosting obtained by dividing the first $n$ levels of the $d$-ary tree into two groups of consecutive levels and treating the $n$-stage game as two consecutive games on $k$ and $n-k$ stages respectively.
The second step uses the structure of the underlying graph to show that $\mathbb E[V_n]$ satisfies a certain subadditivity assumption, which allows us to conclude that $(\mathbb E[V_n])_{n\ge 1}$ converges to a constant $v_{\infty}$, and moreover, $|\mathbb E[V_n]-v_{\infty}|$ is polynomially small. The proof of Proposition~\ref{prop:oriented} relies on a simple explicit construction.

\paragraph{Perspectives}
The proofs of Theorems \ref{thm:1} and \ref{thm:2} have a similar structure but use different arguments. 
A challenging research question would be to prove convergence of $(V_n)$ and concentration bounds in any weakly transitive directed game, irrespective of the expansion speed of the underlying graph, thus unifying Theorems \ref{thm:1} and \ref{thm:2}. 

\paragraph{Plan of the paper.} This paper is organized as follows. 
In Section~\ref{sec:prelims}, we state some classical results we will use later.
Then, in Section~\ref{sec:thm1}, we prove Theorem~\ref{thm:1}, and in Section~\ref{sec:thm2}, we prove Theorem~\ref{thm:2}.
In Section~\ref{sec:prop}, we prove Proposition~\ref{prop:oriented}.


\section{Classical results}\label{sec:prelims}
In our proofs, we make use of the well-known \emph{bounded difference inequality}, also known as \emph{McDiarmid's inequality}, tightly related to \emph{Azuma's inequality}.



\begin{lemma}[Corollary~2.27 in~\cite{JLR00}]\label{lem:McDiar}
    Fix a function $f \colon \Lambda_1 \times \dots \times \Lambda_N \rightarrow \mathbb{R}$ and let 
    $Y_1, \dots, Y_N$ be independent random variables taking values in $\Lambda_1, \dots, \Lambda_N$, respectively. 
    Suppose that there are positive constants $c_1, \dots, c_N$ such that, for every two vectors $z, w \in \Lambda_1\times \dots\times \Lambda_N$ that differ only in the $k$-th coordinate, we have $|f(z) -f(w)| \leq c_k$. 
    Then, for every $t \ge 0$, the random variable $X = f(Y_1,\dots,Y_N)$ satisfies
    \[
        \mathbb{P}(X - \mathbb{E}[X] \ge t) \leq \exp\left( -\frac{t^2}{2\sum_{i=1}^N c_i^2}\right) \,.
    \]
    \[
        \mathbb{P}(X - \mathbb{E}[X] \leq -t) \leq \exp\left( -\frac{t^2}{2\sum_{i=1}^N c_i^2}\right) \,.
    \]
\end{lemma}

We also use the following result that states convergence of almost subadditive sequences.

\begin{lemma}[Theorem 23 in~\cite{BE52}]\label{lem:subadd}
Fix an increasing function $\phi \colon \NN \to (0, \infty)$ such that the sum of $(\phi(n)/n^2)_{n\ge 1}$ is finite, and a function $f \colon \NN \to \RR$ such that, for all $n \in \NN$ and all integers $m \in [n/2, 2n]$, $f(n + m) \le f(n) + f(m) + \phi(n + m)$. Then, there exists $\ell \in \RR\cup \{-\infty\}$ such that
\[\frac{f(n)}{n} \xrightarrow[n \to \infty]{} \ell.\]
\end{lemma}

\section{\texorpdfstring{$\delta$-transient games: proof of Theorem~\ref{thm:1}}{}}\label{sec:thm1}

Fix an initial state $z_0$ and write $V_n = V_n(z_0), Z_n = Z_n(z_0)$ for short. 
To begin with, we show that $V_n$ is well concentrated around its expected value. 
Note that the next lemma holds for weakly transitive games in general and will be reused in the next section.

\begin{lemma}\label{lem:exp conc 1}
    For every $t \ge 0$, 
    \[
        \PP(V_n - \EE[V_n]\ge t)\le \exp\left(-\frac{t^2n^2}{2 h(n)}\right) \,,
    \]
    \[
        \PP(V_n - \EE[V_n]\le -t)\le \exp\left(-\frac{t^2n^2}{2 h(n)}\right) \,.
    \]
\end{lemma}
\begin{proof}
    Define the (random) vectors $X_k = (G_z)_{z\in Z_k}\in [0,1]^{|Z_k|}$. Then, since $Z^{(n)}\subseteq Z_{[h(n)]}$, $V_n$ can be written as $f(X_1, \dots, X_{h(n)})$ for some function $f \colon [0,1]^{|Z_1|}\times \dots\times [0,1]^{|Z_{h(n)}|}\to \RR$. Moreover, for every integer $k\in [1, h(n)]$, the token visits the set $Z_k$ at most once and therefore, for every pair of strategies $(\sigma, \tau)\in \Sigma\times \mathcal{T}$, $\gamma_n^{z_0}(\sigma, \tau)$ varies by at most $1/n$ as a function of $X_k$. Hence, for every choice of vectors $(x_i)_{i=1}^{h(n)}\in [0,1]^{|Z_1|}\times \dots\times [0,1]^{|Z_{h(n)}|}$ and $x_k'\in [0,1]^{|Z_k|}$,
    \[
        |f(x_1, \dots, x_k, \dots, x_{h(n)}) - f(x_1, \dots, x_k', \dots, x_{h(n)})|
            \le \frac{1}{n} ',.
    \]
    Lemma~\ref{lem:McDiar} applied to $V_n$ finishes the proof. 
\end{proof}

In the remainder of the proof, we show that $\mathbb E[V_n]$ converges to a constant polynomially fast. Next, we state and prove an auxiliary lemma relating the values of games of different lengths.

\begin{lemma}\label{lem:differentstages}
    Fix integers $n\ge 1$ and $k\in [1,n]$. Then, $|V_n-V_{n-k}| \le k/n$.
\end{lemma}
\begin{proof}
    Suppose that Player $1$ (resp. Player $2$) plays the first $n-k$ stages according to an optimal strategy for the $(n-k)$-stage game, and plays arbitrarily during the remaining $k$ stages of the $n$-stage game.
    Then, $nV_{n} \ge (n-k)V_{n-k}$ and $nV_{n} \le (n-k)V_{n-k} + k$. Hence, $|n(V_n - V_{n-k})|\le \max(kV_{n-k}, k-kV_{n-k})\le k$, which implies the statement of the lemma.
\end{proof}

The next lemma shows that starting from different initial states 
changes the $n$-value only slightly when $n$ is large.

\begin{lemma}\label{lem:z_0 = z^*}
    For every $z\in Z$, $|\mathbb E[V_n(z)] - \mathbb E[V_n(z^*)]| = O(n^{-\delta})$.
\end{lemma}
\begin{proof}
    Denote by $E$ the set of states $z\in Z^{(M)}$ that are equivalent to $z^*$. 
    By Definition~\ref{def:wt}, independently of the moves of Player 2, $E \neq \emptyset$, and Player 1 can ensure that the token is at a state in $E$ after an even number of $\ell\le M$ stages. Hence, using Lemma~\ref{lem:differentstages}, we have
    \begin{equation}\label{eq:mineq}
        nV_n
            \ge (n-M)\min_{z\in E} V_{n-M}(z)
            \ge (n-M)\min_{z\in E}(V_n(z)-M/n)
            \ge \min_{z\in E} nV_n(z) - 2M \,.
    \end{equation}
    Now, we bound from below the expectation of the right-hand side. 
    Let $\Delta$ be the maximum out-degree of $\Gamma$. 
    Then, $|E|\le |Z^{(M)}|\le 1+\Delta+\ldots+\Delta^M\le (M+1)\Delta^M$ together with the choice of $\eps_n$ from Definition~\ref{def1} imply that
    \begin{equation}\label{eq:minexp}
        \begin{split}
            \mathbb E\left[\min_{z\in E} V_n(z)\right]
                &\ge (\mathbb E[V_n(z^*)] - \eps_n) (1-\mathbb P(\exists z\in E: V_n(z)\le \mathbb E[V_n(z)]-\eps_n))\\
                &\ge (\mathbb E[V_n(z^*)] - \eps_n) (1-(M+1)\Delta^M \psi(n, \eps_n)) = \mathbb E[V_n(z^*)] - O(n^{-\delta}) \,,
        \end{split}
    \end{equation}
    where the second inequality comes from a union bound and the last equality is implied by the fact that $\eps_n+(M+1)\Delta^M\psi(n, \eps_n) = O(n^{-\delta})$. 
    Thus, taking expectations on both sides of~\eqref{eq:mineq} and using~\eqref{eq:minexp} shows that 
    \begin{equation}\label{eq:UB}
        \mathbb E[V_n]\ge \mathbb E[V_n(z^*)] - O(n^{-\delta}+2M/n) 
            = \mathbb E[V_n(z^*)] - O(n^{-\delta}) \,.
    \end{equation}
    
    Similarly, Player 2 can ensure that the token reaches a state in $E$ after an even number of $\ell\le M$ stages. 
    Hence,
    \begin{equation}\label{eq:maxeq}
        nV_n\le (n-M)\max_{z\in E} V_{n-M}(z)+M
            \le \max_{z\in E} nV_n(z)+M \,.
    \end{equation}
    At the same time, similarly to~\eqref{eq:minexp}, $\mathbb E\left[\max_{z\in E} V_n(z)\right]$ is bounded from above by
    \begin{equation*}
        (\mathbb E[V_n(z^*)] + \eps_n) (1-\mathbb P(\exists z\in E: V_n(z)
            \ge \mathbb E[V_n(z)]+\eps_n)) + \mathbb P(\exists z\in E: V_n(z)\ge \mathbb E[V_n(z)]+\eps_n) \,,
    \end{equation*}
    which is at most $\mathbb E[V_n(z^*)] + (\eps_n+(M+1)\Delta^M\psi(n, \eps_n)) = \mathbb E[V_n(z^*)] + O(n^{-\delta})$. Combining this with~\eqref{eq:maxeq} shows that $\mathbb E[V_n]\le \mathbb E[V_n(z^*)]+O(n^{-\delta})$, and together with the upper bound in~\eqref{eq:UB} this finishes the proof.
\end{proof}

Next, we show that the expected value of $V_n$ converges as $n\to \infty$.

\begin{lemma}\label{lem:conv exp 1}
    There is a constant $v_{\infty}$ independent of the initial state such that $|\EE[V_n] - v_{\infty}| = O(n^{-\delta})$ as $n\to \infty$.
\end{lemma}
\begin{proof}
    By Lemma~\ref{lem:z_0 = z^*}, it is sufficient to show the lemma assuming $z_0 = z^*$.
    First, we show that $\EE[V_n]$ converges to a limit $v_{\infty}\in \mathbb R$ as $n\to \infty$. 
    By Lemma~\ref{lem:exp conc 1} and a union bound, for all $t \ge 0$,
    \begin{equation}\label{eq:exists z}
    \begin{split}
        \PP(\exists z \in Z^{(2n)}, |V_n(z) - \EE[V_n(z)]| \ge t)
        &\le \sum_{z \in Z^{(2n)}} \PP(|V_n(z) - \EE[V_n(z)]| \ge t)\\
        &\le 2 \exp \left(-\frac{t^2 n^2}{2h(n)} \right) \max_{z\in Z} |Z^{(2n)}(z)| = 2\psi(n, t) \,.
    \end{split}
    \end{equation}
        
    By definition of $\delta$-transient game, there exists $(\eps_n)_{n \in \NN}$ such that $\eps_n + \psi(n, \eps_n) = O(n^{-\delta})$. 
    Denote by $E$ the set of vertices in $Z^{(2n)}$ that are equivalent to~$z^*$.
    Now, Lemma~\ref{lem:z_0 = z^*} implies that there is a constant $K'>0$ such that, for every $n\ge 1$, $|\mathbb E[V_n(z)]-\mathbb E[V_n]|\le K'n^{-\delta}$. 
    Combining this with~\eqref{eq:exists z}, we get that
    \begin{align*}
        \PP \left( \min_{z \in Z^{(2n)}} V_n(z) 
            \le \EE[V_n] - \eps_n - K'n^{-\delta} \right) 
            &\le \PP \left( \min_{z \in Z^{(2n)}} |V_n(z) - \EE[V_n(z)]|\ge \eps_n \right) \\
            &\le \PP(\exists z \in E, |V_n(z) - \EE[V_n(z)]| \ge \eps_n) \\
            &\le 2\psi(n, \eps_n) = O(n^{-\delta}) \,.
    \end{align*}
    \noindent
    In particular, it follows directly that
    \begin{align*}
        \EE \left[ \min_{z \in Z^{(2n)}} V_n(z) \right] 
            &\ge \left( \EE[V_n] - \eps_n - K'n^{-\delta} \right) \, \PP \left(\min_{z \in Z^{(2n)}} V_n(z) \ge \EE[V_n] - \eps_n - K'n^{-\delta} \right) \\
            &\ge \left( \EE[V_n] - \eps_n - K'n^{-\delta} \right) \, (1-2\psi(n, \eps_n))
            \ge \EE[V_n] - 2(\psi(n, \eps_n) + \eps_n) - K'n^{-\delta} \,.
    \end{align*}
    
    Now, fix an integer $m \in [1,2n]$ and consider the $(m+n)$-stage game. 
    Suppose that Player $1$ plays according to an optimal strategy for the $m$-stage game up to stage $m$ and, once the $m$-stage game terminates at a state $z_m$, continues to play according to an optimal strategy for the subsequent $n$-stage game. 
    Note that $z_m\in Z^{(2n)}$, so the above strategy of Player $1$ for the first $m+n$ steps guarantees a gain of $\tfrac{m}{m+n}V_m + \tfrac{n}{m+n}\min_{z \in Z^{(2n)}} V_n(z)$. Thus,
    \begin{equation}\label{eq:mn}
        \begin{split}
            (m + n) \EE[V_{m + n}] 
                &\ge m \EE[V_m] + n \EE\left[\min_{z \in Z^{(2n)}} V_n(z)\right] \\
                &\ge m \EE[V_m] + n \EE[V_n] - 2n(\psi(n, \eps_n) + \eps_n) - K'n^{1-\delta} \,.
        \end{split}
    \end{equation}
    
    Since $\psi(n, \eps_n) + \eps_n = O(n^{-\delta})$, there is a constant $K'' > 0$ such that, for all $n \ge 1$,
    \[
        2n(\psi(n, \eps_n) + \eps_n)+K'n^{1-\delta}
            \le 2K''n^{1-\delta} \,.
    \]
    Thus, using Lemma~\ref{lem:subadd} with $f \colon n \mapsto -n \EE[V_n]$ and $\phi \colon n\mapsto 2K'' n^{1 - \delta}$ (note that $\phi$ is increasing and $\sum_{n\ge 1} \phi(n) / n^2 = 2K'' \sum_{n\ge 1} 1 / n^{1 + \delta} < \infty$) implies that $\EE[V_n]$ converges to a limit $v_{\infty}\in \mathbb R\cup \{\infty\}$ as $n\to \infty$.
    Note that $v_{\infty}$ is in $[0,1]$ since this is the support of all payoff variables.
    
    Finally, using~\eqref{eq:mn} with $m = n$, for every $n\ge 1$, we have that
    \[
        \EE[V_{2n}] 
            \ge \EE[V_n] - (\psi(n, \eps_n) + \eps_n) - \frac{K'n^{-\delta}}{2}
            \ge \EE[V_n] - K'' n^{-\delta} \,.
    \] 
    In particular, for all integers $\ell, n \ge 1$, iterating the above observation for $n, 2n, \dots, 2^{\ell-1}n$ gives that
    \begin{equation}\label{eq:telescope}
        \EE[V_{2^\ell n}]
            \ge \EE[V_n] - K'' n^{-\delta} \sum_{j=0}^{\ell-1} 2^{-\delta j}
            \ge \EE[V_n] - \frac{K''}{1-2^{-\delta}}n^{-\delta} \,.
    \end{equation}
    Taking $\ell \to \infty$, we conclude that $v_\infty \ge \EE[V_n] - O(n^{-\delta})$. A similar reasoning exchanging Player $1$ with Player $2$ shows that $v_\infty \le \EE[V_n] + O(n^{-\delta})$ and concludes the proof of the lemma.
\end{proof}

Finally, we are ready to prove Theorem~\ref{thm:1}.

\begin{proof}[Proof of Theorem~\ref{thm:1}]
    Fix an arbitrary $\eps > 0$. 
    By Lemma~\ref{lem:conv exp 1}, there is a constant $K > 0$ such that $|v_\infty - \EE[V_n]|\le K n^{-\delta}$ for all $n\ge 1$, independently of the initial state. Combining this with the triangle inequality and Lemma~\ref{lem:exp conc 1} shows that, for every $t\ge 0$,
    \begin{align*}
        \PP( |V_n - v_\infty| \ge t + K n^{-\delta} ) 
            &\le \PP( |V_n - \EE[V_n]| \ge t + K n^{-\delta} - |\EE[V_n] - v_\infty|) \\
            &\le \PP( |V_n - \EE[V_n]| \ge t) \le 2\exp \left( \frac{-t^2 n^2}{2h(n)} \right) \,,
    \end{align*}
    which is the desired result.
\end{proof}

\section{\texorpdfstring{Directed games on trees: proof of Theorem~\ref{thm:2}}{}}
\label{sec:thm2}

The first lemma in this section bootstraps upon the conclusion of Lemma~\ref{lem:exp conc 1} (which still holds in this setting), thus deriving superexponential concentration for the value of the $n$-stage game. Below, $\log$ stands for the natural logarithm.

\begin{lemma}\label{lem:bootstrap}
    Fix $\delta\in (0,1/2)$ and $t\ge n^{-\delta}$. For every integer $n\ge 1$ and even integer $k\in [2,n]$ such that 
    \begin{equation}\label{ineqcondition}
        k \log d + 2 \log 2 \leq t^2 (n-k),
    \end{equation}
    we have 
    \begin{align*}
        \PP (nV_n - (n-k)\EE[V_{n - k}]  
            \ge\;\;\, (n-k) t + k) 
            \le \exp\left(-\frac{d^{k/2}}{6}\right) \,, \\
        \PP (nV_n - (n-k)\EE[V_{n - k}]  
            \le -(n-k)t-k) 
            \le \exp\left(-\frac{d^{k/2}}{6}\right) \,.
    \end{align*}
\end{lemma}
\begin{proof}
    First of all, since $T$ is a transitive graph, for all $n\ge 1$, $(V_n(z))_{z\in Z}$ have the same distribution.	For every even integer $k\in [n]$, denote
    \[
    	S_k \defas \{ z \in Z_k : V_{n-k}(z)-\EE[V_{n-k}] \geq   t \} \,.
    \]
    In other words, $S_k$ is the set of vertices that could be reached from $z_0$ after $k$ stages, for which the value of the $(n-k)$-stage game starting at $z$ is greater than or equal to $\EE[V_{n-k}] +t$.
	
	Define the event $\mathcal{E}_k \defas \{|S_k| \geq  d^{k/2}\}$. 
	We provide an upper bound for $\PP(\mathcal{E}_k)$.
	Since the random variables $(V_{n-k}(z))_{z \in Z_k}$ are i.i.d., we have that $|S_k|$ follows a binomial distribution $\mathrm{Bin}(d^k, q)$ where $q \defas \PP(V_{n-k} \geq \EE[V_{n-k}] + t)$.
	Consequently, by Lemma~\ref{lem:exp conc 1} (where $h(n) = n$ is the transient speed of the family of partitions $(\Pi_z)_{z\in Z}$ where, for all $z\in Z$ and $k\ge 2$, $Z_k(z)$ contains all descendants of $z$ at distance $k$), $|S_k|$ is stochastically dominated by a binomial random variable $\mathrm{Bin}(d^k, \tilde{q})$ where $\tilde{q} = \exp(-t^2(n-k)/2)$. 
    In particular,  
	\begin{equation*}
    	\PP(\mathcal{E}_k)
        	\le \PP \left( \mathrm{Bin}(d^k, \tilde{q}) \ge d^{k/2} \right) \,.
	\end{equation*} 
	The random variable $\mathrm{Bin}(d^k, \tilde{q})$ has mean $\mu \defas d^k \tilde{q}$. 
    We define 
    \[
        \xi 
            \defas \frac{d^{k/2}}{\mu} - 1 
            = \exp\left(\frac{t^2(n-k) - k\log(d)}{2}\right) - 1
            \ge 1 \,,
    \]
	where the last inequality comes from~\eqref{ineqcondition}. 
    Since $d^{k/2} = (1 + \xi) \mu$, we have that 
	\[
    	\PP( \mathrm{Bin}(d^k, \tilde{q}) \ge d^{k/2} )
            = \PP( \mathrm{Bin}(d^k, \tilde{q}) 
        	\ge (1+\xi)\mu ) \,.
	\]
	Therefore, since $\xi \geq 1$ (so $3\xi\ge 2+\xi$), by Chernoff's bound,
	\begin{align*}
        \PP \left( \mathrm{Bin}(d^k, \tilde{q}) 
        \ge d^{k/2} \right) 
            \le \exp\left(-\frac{\xi^2\mu}{2+\xi}\right)
            \le \exp\left(-\frac{\xi\mu}{3}\right) 
            = \exp\left(-\frac{d^{k/2}}{3}  \left(1-d^{k/2}\tilde{q}\right) \right) \,.
	\end{align*}
	Since $\xi = 1/(d^{k/2}\tilde{q}) - 1\ge 1$, we have that $1-d^{k/2}\tilde{q} \geq 1/2$, which finally yields 
	\begin{equation}\label{finalchernoff}
        \PP(\mathcal{E}_k)  
            \leq \exp\left(-\frac{d^{k/2}}{6} \right) \,.
	\end{equation}
    At the same time, on the event $|S_k| < d^{k/2}$ (that is, $\overline{\mathcal{E}_k}$), Player $2$ can ensure that the token avoids ending up in $S_k$ after $k$ stages. 
    Indeed, at each of the $k/2\in \mathbb N$ turns corresponding to decisions of Player $2$, by the pigeonhole principle, Player $2$ can always move the token to a vertex having at most a $(1/d)$-fraction of all remaining elements in $S_k$ among its descendants. 
    Since Player $2$ has $k/2$ turns and $d^{-k/2} |S_k| < 1$, Player $2$ can safely avoid the set $S_k$ at stage $k$. 

	Let us condition on the event $\overline{\mathcal{E}_k}$. 
    Then, Player $2$ can guarantee that the sum of the payoffs over the last $n-k$ stages is strictly smaller than $(n-k)(\EE[V_{n-k}] + t)$. 
    Moreover, the sum of the first $k$ payoffs is at most $k$. 
    Consequently, Player $2$ can guarantee that, after $n$ stages, the global mean payoff is strictly smaller than $k/n + (n-k)(\EE[V_{n-k}] + t)/n$, in other words, 
    \begin{equation}\label{eq:strict}
        nV_n < (n-k) \EE[V_{n - k}] + (n-k) t + k \,.
    \end{equation}
    In particular, using~\eqref{finalchernoff} implies that
	\[
        \PP\left( nV_n - (n-k) \EE[V_{n - k}] \geq (n-k) t + k \right)
            \leq \PP(|S_k| \geq d^{k/2}) 
            = \PP(\mathcal{E}_k)
            \le \exp\left(-\frac{d^{k/2}}{6}  \right) \,.
    \]
	
	A similar reasoning for Player $1$ (using the sets $\tilde{S}_k \defas \{ z \in Z_k : V_{n-k}(z)-\EE[V_{n-k}] \leq -t \}$ instead of $S_k$ and replacing~\eqref{eq:strict} with $nV_n > (n-k)\mathbb E[V_{n-k}] - (n-k)t$) yields
    \[
        \PP\left(nV_n - (n-k)\EE[V_{n - k}] \leq -(n-k)t\right)
            \leq \exp\left(-\frac{d^{k/2}}{6} \right) \,,
    \]
    which implies the second statement.
    Note that the additional $-k$ in it is introduced for reasons of symmetry only.
\end{proof}

Next, we show that the expected value of the $n$-stage game converges rapidly as $n$ grows to infinity. 

\begin{lemma}\label{lem:conv exp 2}
	There exists $v_\infty \in \RR$ such that, for every $\delta\in (0,1/2)$, we have $|\EE[V_n] - v_\infty| = O(n^{-\delta})$ as $n\to \infty$.
\end{lemma}
\begin{proof}
	Fix $\delta' \in (0,1/2)$ and $t\ge n^{-\delta'}$. For each $n \geq 1$, we set $k = k(n) \defas 2 \big\lfloor n^{1-2\delta'}/4\log d \big\rfloor$. 
    Then, $k\log d + 2 \log 2 \leq t^2 (n-k)$ for all large $n$. 
    For every even integer $m \in [n/2,2n]$ and large $n$, we have 
	\begin{align*}
        \PP \left( \min_{z \in Z_{m}} nV_n(z) \leq (n-k)(\EE[V_{n-k}]-t) - k \right) 
            &\le \sum_{z \in Z_{m}} \PP \left(nV_n(z) \leq (n-k)(\EE[V_{n-k}]-t) - k \right) \\
            &\le d^m \exp\left(-\frac{d^{k/2}}{6} \right) \\
            &\le \exp \left(2n\log d -\frac{d^{\lfloor n^{1-2\delta'}/4\log d\rfloor}}{6} \right) \,,	
    \end{align*}
     where the first inequality comes from a union bound and the second inequality comes from \Cref{lem:bootstrap}. 
     Fix $\delta \in (0, \delta')$ and define, for all $n \geq 1$,
    \[
        \eps_n \defas n^{-\delta}
        \quad\text{and}\quad
        \psi(n) \defas \exp \left(2n\log d -d^{\lfloor n^{1-2\delta'}/4\log d\rfloor}/6 \right) \,.
    \] 
    For large $n$ and every even integer $m \in [n/2, 2n]$, we have
    \begin{align}
    	\EE \left[ \min_{z \in Z_{m}} V_n(z) \right]
            &\ge \left(\frac{n-k}{n}(\EE[V_{n-k}]-\eps_n) - \frac{k}{n}\right) \PP \left(\min_{z \in Z_{m}} nV_n(z) > (n-k)(\EE[V_{n-k}]-\eps_n) - k \right)\nonumber \\
            &\ge \left(\frac{n-k}{n}(\EE[V_{n-k}]-\eps_n) - \frac{k}{n}\right) (1- \psi(n))\nonumber \\
            &\ge \left(\EE[V_{n-k}] - \frac{k}{n}(1+\EE[V_{n-k}]) - \eps_n \right) (1 - \psi(n))\nonumber\\
            &\ge \left(\EE[V_{n}] - \frac{3k}{n} - \eps_n \right) (1 - \psi(n)) \ge \EE[V_n] - (\psi(n) + 2\eps_n) \,, \label{eq:minE}
	\end{align}
	where in the fourth inequality we used that $\mathbb E[V_n]\le \mathbb E[V_{n-k}]+k/n$ by Lemma \ref{lem:differentstages} and $1+\mathbb E[V_{n-k}]\le 2$, and the last inequality is valid for large $n$ because $k/n = o(\eps_n)$.
	
    Consider integers $n \geq 1$ and even $m\in [n/2,2n]$. In the $(n+m)$-stage game, Player $1$ can play according to an optimal strategy for the $m$-stage game starting at $z_0$, and then play according to an optimal strategy for the $n$-stage game starting from the state $z$ reached after $m$ stages.
    This guarantees that $(m+n)V_{m+n}\ge m V_m + \min_{z \in Z_m} n V_n(z)$. Taking expectations on both sides and using~\eqref{eq:minE} yields
	\begin{align*}
    	(m + n) \EE[V_{m + n}] 
        \ge m \EE[V_m] + n \EE \left[ \min_{z \in Z_m} V_n(z) \right]
        \ge m \EE[V_m]+ n\EE[V_n] - n(\psi(n) + 2\eps_n) \,.
	\end{align*}

	We find a similar inequality for odd $m\in [n/2,2n]$. 
    In this case, $m+1$ is even and also in $[n/2,2n]$. 
    Then, the previous inequality applied to $m+1$ and $n$ yields
    \begin{equation}\label{eq:m+1}
        (m+n+1) \EE[V_{m+n+1}] 
            \ge (m+1) \EE[V_{m+1}]+ n\EE[V_n] - n(\psi(n) + 2\eps_n) \,.
    \end{equation}
    However, 
	\[
        (m+n) \EE[V_{m+n}] 
            \geq (m+n+ 1) \EE[V_{m+n+1}]-1
        \quad\text{and}\quad 
        (m+1) \EE[V_{m+1}] 
            \geq m \EE[V_{m}] \,,
    \]
    which combined with~\eqref{eq:m+1} gives  
	\[
        (m + n) \EE[V_{m+n}] \geq m \EE[V_{m}]+ n\EE[V_n] - n(\psi(n) + 2\eps_n)-1 \,.
    \]
    To sum things up, for large $n$ and $m\in [n/2,2n]$,
	\begin{equation}\label{expectineq}
        (m + n) \EE[V_{m  +n}] \geq m \EE[V_{m}]+ n\EE[V_n] - n(\psi(n) + 2\eps_n) - 1 \,.
	\end{equation}
	Recall that there is a constant $K' > 0$ such that, for all $n\ge 1$, $n(\psi(n) + 2\eps_n)+1 \le K' n^{1-\delta}$. We define $\phi(n) \defas K' n^{1-\delta}$ and deduce from~\eqref{expectineq} that 
	\[
	   (m + n) \EE[V_{m  +n}] \geq m \EE[V_{m}]+ n\EE[V_n] - \phi(n+m) \,.
	\]
	Moreover, $\phi$ is increasing and verifies  $\sum_{n\ge 1} \phi(n)/n^2 < \infty$. 
    Consequently, Lemma \ref{lem:subadd} applied to the function $f \colon n\in \mathbb N\mapsto -n\mathbb E[V_n]$ implies that that $\EE[V_n]$ converges to a limit $v_{\infty}\in \mathbb R\cup \{\infty\}$ as $n\to \infty$.
    Note that $v_{\infty}\in [0,1]$ since $V_n\in [0,1]$ for all $n\ge 1$. 
		
	Finally, using~\eqref{expectineq} with $m = n$ and a telescopic summation shows that the inequality~\eqref{eq:telescope} still. 
    In particular, we conclude that $v_\infty\ge \EE[V_n] - O(n^{-\delta})$. 
    A similar reasoning replacing Player $1$ with Player $2$ shows that $v_\infty\le \EE[V_n] + O(n^{-\delta})$ and concludes the proof of the lemma.
\end{proof}

We are now ready to prove Theorem~\ref{thm:2}.

\begin{proof}[Proof of Theorem~\ref{thm:2}]
	Fix $t\ge n^{-\delta}$ and let $K'$ be a constant such that $|\mathbb E[V_n] - v_{\infty}|\le K' n^{-\delta}$ for all large $n$. Using that, for all $n$ and $k\le n$, we have $|nV_n - (n-k)V_{n-k}|\le k$, and fixing $k = 2\lceil \tfrac{t^2n}{4\log d}\rceil$ (which satisfies~\eqref{ineqcondition}), we get
	\begin{align*}
    	&\PP(|V_n - v_\infty|\ge t+2t^2+K' n^{-\delta}) \\
            &\qquad \le \PP\left(\left|V_n - \frac{n-k}{n}\EE[V_{n-k}]\right|\ge t+2t^2+K'n^{-\delta} - \left|\frac{n-k}{n}\EE[V_{n-k}] - \EE[V_n]\right| - |\EE[V_n]-v_\infty|\right) \\
            &\qquad \le \PP\left(\left|V_n - \frac{n-k}{n}\EE[V_{n-k}]\right|\ge t + t^2 \right) \\
            &\qquad \le \PP(|nV_n - (n-k)\EE[V_{n - k}]| \ge (n-k)t + k) \\
            &\qquad \le \exp\left(-\frac{d^{\lfloor k/2\rfloor}}{6} \right) \le \exp\left(-\frac{d^{t^2n/(4\log d)}}{6} \right) 
            = \exp\left(-\frac{1}{6}\exp\left(\frac{t^2n}{4}\right)\right) \,,
    \end{align*}
    where the first inequality comes from the triangle inequality, the second inequality comes from the definition of $K'$ and the fact that $|nV_n-(n-k)V_{n-k}|\le k\le nt^2$, and the third inequality once again uses the fact that $k\le nt^2$. 

    Finally, choosing $K\ge K'$ sufficiently large ensures that, first, the upper bound shown above holds for all $n\ge 1$ (and not only for large $n$), and second, the upper bound holds for all $t\ge 0$, which finishes the proof.
\end{proof}

\section{\texorpdfstring{Oriented directed games: proof of Proposition~\ref{prop:oriented}}{}}\label{sec:prop}

We present a simple and self-contained proof of Proposition~\ref{prop:oriented}.

\begin{proof}
    First, by density of the rational vectors in $\RR^d$ and rescaling, we may assume that $u \in \mathbb Z^d$ is such that the greatest common divisor of its coordinates is 1. 
    Then, for every integer $i \ge 1$ and initial state $z_0 = z$, defining  $Z_{2i}(z) \defas \{ w \in Z : w \cdot u = z\cdot u+i \}$, $Z_{2i+1}(z) \defas \{ w \in Z : w\cdot u = z \cdot u - i \}$, and $Z_1(z) \defas \{ w \in Z \setminus \{ z \} : w \cdot u = z \cdot u \}$ shows that the game is directed. 
    Indeed, $(Z_i(z))_{i\ge 0}$ form a partition of $Z$ for all $z\in Z$, and each of them could be visited at most once by the token.

    Now, fix $\delta \in (0, 1/2)$ and $z_0 = z \in Z$.
    To see that the game is $\delta$-transient, set $r = \max_{ u v \in E(\Gamma) } \| v - u \|_2$.
    After $n$ steps of the process, the position $z_n$ of the token satisfies $\| z_n - z \|_2 \le nr$, and by the Cauchy-Schwarz inequality, 
    \[
        |(z_n - z) \cdot u | 
            \le \| z_n - z \|_2 \cdot \| u \|_2 
            \le \lceil n r \cdot \| u \|_2 \rceil \safed M 
            = M(n) \,.
    \]
    In particular, $Z^{(n)}(z)$ is contained in the ball with radius $M$ around $z$, which itself is contained in $Z_{[2M+1]}(z)$, so the transient speed of the process satisfies $h(n)\le 2M(n)+1$ for all $n\ge 1$. 
    Finally, take $\delta \in (0, 1/2)$ and set $\eps_n \defas n^{-\delta}$. 
    Then,
    \begin{align*}
        \psi(n, \eps_n) 
            &= \exp\left( -\frac{\eps_n^2 n^2}{2h(n)} \right) \max_{z\in Z}|Z^{(2n)}(z)|\\
            &\leq \exp\left( -\frac{\eps_n^2 n}{6 r\cdot \| u \|_2} \right) (2 n r \cdot \| u \|_2 + 1)^d \\
            &= \exp\left( -\frac{n^{1 - 2\delta}}{6 r\cdot \| u \|_2} \right) (2n r \cdot \| u \|_2 + 1)^d 
            = O(n^{-\delta}) \,.
    \end{align*}
    Hence, for all $\delta \in (0, 1/2)$, $\eps_n + \psi(n, \eps_n) = O(n^{-\delta})$, and therefore, the game is $\delta$-transient. 
\end{proof}

\section*{Acknowledgments}

This work was supported by the French Agence Nationale de la Recherche (ANR) under references ANR-21-CE40-0020 (CONVERGENCE project) and ANR-20-CE40-0002 (GrHyDy), and by Fondecyt grant 1220174. This collaboration was mainly conducted during a 1-year visit of Bruno Ziliotto to the Center for Mathematical Modeling (CMM) at University of Chile in 2023, under the IRL program of CNRS.

\bibliographystyle{plain}
\bibliography{references}

\end{document}